\documentclass [12pt]{article}
\usepackage{graphicx,amssymb,amsfonts,latexsym,amsmath,amsthm,times}
\usepackage{epsfig}
\usepackage{fancyhdr}
\usepackage{color}
\usepackage[english]{babel}
\numberwithin{equation}{section}

\newtheorem{theorem}{Theorem}

\newtheorem{definition}[theorem]{Definition}

\newtheorem{lemma}[theorem]{Lemma}

\newtheorem{proposition}[theorem]{Proposition}
\newtheorem{remark}[theorem]{Remark}

\usepackage{mathrsfs}
\usepackage{cases}
\usepackage{amscd}
\usepackage{latexsym}
\usepackage{hyperref}
\usepackage[small,nohug,heads=vee]{diagrams}
\diagramstyle[labelstyle=\scriptstyle]

\newcommand{\be}{\begin{equation}}
\newcommand{\ee}{\end{equation}}

\newcommand{\Z}{{\mathbb Z}}
\newcommand{\R}{{\mathbb R}}

\newcommand{\p}{\parallel}
\newcommand{\rarrow}{\rightarrow}

\newcommand{\vphi}{\varphi}
\newcommand{\sig}{\sigma}
\newcommand{\lam}{\lambda}

\newcommand{\om}{\omega}

\begin{document}

\footnotesize {\flushleft \mbox{\bf \textit{Math. Model. Nat.
Phenom.}}}
 \\
\mbox{\textit{{\bf Vol. 5, No. 4, 2010, pp. }}}

\medskip

\thispagestyle{plain}

\vspace*{2cm} \normalsize \centerline{\Large \bf An Exposition of
the Connection between} \centerline{\Large \bf Limit-Periodic Potentials and Profinite Groups}

\vspace*{1cm}

\centerline{\bf Z. Gan \footnote{E-mail: zheng.gan@rice.edu}}

\vspace*{0.5cm}

\centerline{Department of Mathematics, Rice University, 77005
Houston, USA}

\vspace*{1cm}

\centerline{\em Dedicated to the memory of M. Sh. Birman.}


\vspace*{1cm}

\noindent {\bf Abstract.}
We classify the hulls of different limit-periodic potentials and
show that the hull of a limit-periodic potential is a procyclic
group. We describe how limit-periodic potentials can be generated
from a procyclic group and answer arising questions. As an
expository paper, we discuss the connection between limit-periodic
potentials and profinite groups as completely as possible and review
some recent results on Schr\"odinger operators obtained in this
context.
\vspace*{0.5cm}




\vspace*{1cm}

\setcounter{equation}{0}
\section{Introduction}

This expository paper is motivated by recent research on
limit-periodic Schr\"odinger operators \cite{a,dg1,dg2,dg3}. In
\cite{a}, Avila gave a new way to treat limit-periodic potentials by
regarding limit-periodic potentials as generated by continuous
sampling function along the orbits of a minimal translation of a
Cantor group, allowing one to separate the base dynamics and the sampling function.
Damanik and I discussed this connection in \cite{dg1,dg2,dg3}. However, those
discussions weren't very comprehensive, since the goals of these
works (including Avila's \cite{a}) were to study spectral properties
of the associated Schr\"odinger operators, but not the connection
itself. There are still some questions waiting to be clarified. 

By that definition, Cantor groups belong to a larger class of groups,
profinite groups, that will help us to understand Cantor groups better
(we shall see that a Cantor group, that admits a minimal translation, is a procyclic group). 
The main tasks of this paper are: (1)\begin{em} classification of the hulls of
different limit-periodic potentials;
\end{em}
(2)\begin{em}
characterization of the hull of a limit-periodic potential.
\end{em}
As an expository paper, we will cover related questions with the
goal to provide a reference for future study in limit-periodic
Schr\"odinger operators and other related topics.


\vspace*{0.5cm}
\setcounter{equation}{0}
\section{Classification of the Hulls of Limit-Periodic Potentials}

\begin{definition}
Given a topological group $\Omega$, a map $T : \Omega \to \Omega$ is
called a translation if $T (\omega) = \omega \cdot \omega_0$ for some
$\omega_0 \in \Omega$, and $T$ is called minimal if the orbit $\{
T^n(\omega) : n \in \Z \}$ of every $\omega \in \Omega$ is dense in
$\Omega$.
\end{definition}

Let $\sig$ be the left shift operator on $\ell^\infty(\Z)$, that is,
$(\sig(d))_n = d_{n + 1}$ for every $d \in \ell^\infty(\Z).$ Let
$\mathrm{orb}(d) = \{ \sig^k (d) : k \in \Z\}$, and let
$\mathrm{hull}(d)$ be the closure of $\mathrm{orb}(d)$ in
$\ell^\infty (\Z).$

\begin{definition}
A potential $p \in \ell^\infty(\Z)$ is called periodic if
$\mathrm{orb}(p)$ is finite, and $p$ is called limit-periodic if it
belongs to the closure of the set of periodic potentials.
\end{definition}

For a limit-periodic potential $d \in \ell^{\infty}(\Z)$,  it is
easy to see that every $d^{'} \in \mathrm{hull}(d)$ is still
limit-periodic. We also have

\begin{proposition} \label{prop:avilahull} \emph{\cite[Lemma
2.1]{a}}
If $d$ is a limit-periodic potential, then $\mathrm{hull}(d)$ is
compact and has a unique topological group structure with identity
$\sig^0(d) = d$ such that
$$\phi : \Z \longrightarrow  \mathrm{hull}(d),\quad k \longrightarrow \sig^{k}(d)$$
is a homomorphism. Also, the group structure is Abelian and there
exist arbitrarily small compact open neighborhoods of $d$ in
$\mathrm{hull}(d)$ that are finite index subgroups.
\end{proposition}

The last statement in Proposition \ref{prop:avilahull} tells us that
$\mathrm{hull}(d)$ is totally disconnected. The translation $T:
\mathrm{hull}(d) \to \mathrm{hull}(d),\ \omega \to  \omega \cdot
\sig(d) = \sig(\omega)$ is minimal, since $\{ \sig^k(d) : k \in \Z
\}$ is dense in $\mathrm{hull}(d)$.

Let us denote $\mathrm{hull}(d)$ by $\Omega_d$. The dual group,
$\hat{\Omega}_d$, of characters on $\Omega_d$ is naturally a
topological subgroup of the circle group $C$, the multiplicative
group of all complex numbers of absolute value $1$, since the dual
group of $\Z$ is $C$. By taking inverse image under the map $\R \to
C, \alpha \to e^{i\alpha }$, we obtain a subgroup $F_d$ of $\R$,
called the \textit{frequency module} of $d$. $F_d$ is countable
since $\Omega_d$ has a countable dense subset $\mathrm{orb}(d)$, and
it is a $\Z$-module.

\begin{proposition} \label{prop:generator} \emph{\cite[Theorem
A.1.1]{as}}
The frequency module $F_d$, is the $\Z$-module generated by
$$\{ \alpha : \lim_{n \to \infty} \frac{1}{2n} \sum^{n}_{k=-n} d(k) e^{-ik\alpha} \neq 0, \alpha \in \R \}.$$
\end{proposition}

\begin{remark}
Appendix 1 of \cite{as} dealt with continuous Schr\"odinger
operators. However, all the results from there can be directly
applied to discrete Schr\"odinger operators as well, since the
proofs in Appendix 1 of \cite{as} don't require any specific
property only successful in the continuum case.
\end{remark}

\begin{proposition} \label{prop:commondivisor} \emph{\cite[Theorem A.1.3]{as}}
$d \in \ell^{\infty}(\Z)$ is limit-periodic if and only if $F_d$ has the property that any $\alpha, \beta \in F_d$ have a common divisor in $F_d$.
\end{proposition}

Through the Peter-Weyl theorem, we have

\begin{proposition}  \label{prop:limitform} \emph{\cite[Corollary A.1.5]{as}}
If $d \in \ell^{\infty}(\Z)$ is limit-periodic, then there exists a
positive integer set $S_d = \{n_j \}$ satisfying $n_j | n_{j+1}$
such that
$$ d(k) = \sum^{\infty}_ {j = 1} p_j (k),$$
where $p_j \in \ell^{\infty}(\Z)$ are $n_j$-periodic. This
convergence is uniform.
\end{proposition}

\begin{remark} \label{remk:limitperiod}
Given any positive integer set $S_d = \{n_j \}$ satisfying $n_j |
n_{j+1}$, let $p_j (k) = t/n^3_j,\ t = k\ (\mathrm{mod}\ n_j)$. 
Then the potential $d \in
\ell^{\infty}(\Z)$, defined by $ d(k) = \sum^{\infty}_ {j = 1} p_j
(k),$  is limit-periodic.
\end{remark}

Given a limit-periodic potential $d \in \ell^{\infty}(\Z)$ which is
not periodic, $F_d$ is countable and infinitely generated (if it is
finitely generated, $d$ is quasi-periodic or periodic; see Appendix
$1$ of \cite{as}). Denote the set of generators of $F_d$ by $G_d =
\{ 2\pi \alpha_j \}$. By Proposition \ref{prop:commondivisor}, $2
\pi \alpha_1$ and $2 \pi \alpha_2$ have a common divisor. By
Proposition \ref{prop:generator},  $$a = \lim_{n \to \infty}
\frac{1}{2n} \sum^{n}_{k=-n} d(k) e^{-i2k \pi \alpha_1}$$ and $$b =
\lim_{n \to \infty} \frac{1}{2n} \sum^{n}_{k=-n} d(k) e^{-i2k \pi
\alpha_2}$$ are both non-zero. Choose  a periodic potential $p \in
\ell^{\infty}(\Z)$ with
$$\p p - d \p_{\infty} \leq \frac{1}{2} \min(|a|,|b|).$$ It follows
that
$$ \lim_{n \to \infty} \frac{1}{2n} \sum^{n}_{k=-n} p(k) e^{-i2k \pi  \gamma} \neq 0 $$
for $\gamma = \alpha_1,\alpha_2$. So $\frac{2\pi}{h}$ divides both
$2\pi \alpha_1$ and $2\pi \alpha_2$ where $h$ is the period \footnote{In this paper, the period of a periodic 
potential is considered as the minimal period.} of $p$
(note that the frequency module of an $h$-periodic potential is
finitely generated, and there is a common divisor $\frac{2 \pi}{h}$
of the generators; see Appendix $1$ of \cite{as}). So there exists
$n_1 \in \Z^{+}$ such that the greatest common divisor of $2\pi
\alpha_1$ and $2\pi \alpha_2$ is $\frac{2 \pi}{n_1}$, and the
$\Z$-module generated by $\{ 2\pi \alpha_1, 2\pi \alpha_2 \}$ is the
$\Z$-module generated by $\{\frac{2 \pi}{n_1} \}$. Similarly, for
$\alpha_1,\alpha_2$ and $\alpha_3$ there exists a positive integer
$n_2$ such that the $\Z$-module generated by $\{ 2\pi \alpha_1, 2\pi
\alpha_2, 2\pi \alpha_3 \}$ is the $\Z$-module generated by
$\{\frac{2 \pi}{n_2} \}$. Clearly, $n_1 | n_2$. By induction, we can
find an infinite positive integer set $$S_d = \{n_j\} \subset \Z^+$$
such that $n_j | n_{j+1}$ and $F_d$ is generated by $G_d = \left
\{\frac{2\pi}{n_j} : n_j \in S_d \right \}.$

We call $S_d$ a \textit{frequency integer set} of $d$. In general,
we say that $S_d$ is a frequency integer set of $d$ if one can
obtain the frequency module $F_d$ from $S_d$. When $d$ is
limit-periodic, for any $n,m \in S_d$, one always has $n | m$ or $m
| n$. If $d(k) = \sum^{\infty}_ {j = 1} p_j (k),$ where $p_j$ are
$n_j$-periodic, then it is easy to conclude that $S_d = \{ n_j \}$
is a frequency integer set of $d$. Clearly, $d$ may have other
frequency integer sets. By the following theorem, we shall see that
there exists a unique maximal frequency integer set $M_d$ in the
sense that every frequency integer set $S_d$ is contained in $M_d$.

\begin{theorem} \label{thm:subfrequency}
Given a limit-periodic potential $d \in \ell^{\infty}(\Z)$ with an
infinite frequency integer set $S_d$, for any limit-periodic
potential $\tilde{d} \in \ell^{\infty}(\Z)$ with a frequency integer
set $S_{\tilde{d}}$, if  $S_{\tilde{d}}$ is an infinite subset of
$S_d$, then we have $\Omega_d \cong \Omega_{\tilde{d}}.$
\end{theorem}

\begin{proof}
Write $S_{\tilde{d}} = \{n_{j_k} \} \subset S_d = \{n_j \} $. Define
a homomorphism by $\vphi: F_{\tilde{d}} \to F_d, \ \vphi(\frac{2\pi}
{n_{j_k}}) = \frac{2\pi} {n_{j_k}}.$ Clearly, $\vphi$ is injective
since $G_{\tilde{d}} \subset G_d$. For any $\frac{2\pi} {n_{t}}
\notin G_{\tilde{d}}$, choose $n_{j_k} > n_t$, and then there is
$t_k$ such that $n_{j_k} = t_k n_t$ (by the property of $S_d$). We
have $\vphi (t_k \frac{2\pi} {n_{j_k}}) = t_k \frac{2 \pi}{n_{j_k}}
= \frac{2\pi} {n_{t}}$,  implying that $\vphi$ is also surjective.
Thus, $F_d = F_{\tilde{d}}.$  Clearly, $\hat{\Omega}_d \cong
\hat{\Omega}_{\tilde{d}}$, since $\hat{\Omega}_d = \{e^{i \alpha } :
\alpha \in F_d \}$ and $\hat{\Omega}_{\tilde{d}} = \{e^{i \beta } :
\beta \in F_{\tilde{d}} \}$. By the Pontryagin duality theorem,
$\Omega_d \cong \Omega_{\tilde{d}}.$
\end{proof}

Theorem \ref{thm:subfrequency} gives us a way to find the maximal
frequency integer set of $d$. If $S_d$ is a frequency integer set of
$d$, one can add positive integers into $S_d$  to get $M_d = \{m_j
\} \subset \Z^+$ such that  $m_{j+1}/m_j$ is a prime number.
Clearly, $S_d \subset M_d$ and $M_d$ is still a frequency integer
set of $d$. The uniqueness of  $M_d$ follows from Theorem
\ref{thm:subfrequency}.

\begin{theorem} \label{thm:classification}
Given limit-periodic potentials $d,\tilde{d} \in \ell^{\infty}(\Z)$
with frequency integer sets $S_d$ and $S_{\tilde{d}}$ respectively,
$\Omega_d \cong \Omega_{\tilde{d}}$ if and only if for any $n_i \in
S_d$ there exists $m_j \in S_{\tilde{d}}$ such that $n_i | m_j$, and
vice versa.
\end{theorem}

\begin{proof}
If $S_d, S_{\tilde{d}}$ are both finite sets, $d, \tilde{d}$ are
periodic, and then the statement follows readily. We assume that
$S_d, S_{\tilde{d}}$ are infinite sets. First assume that for any
$n_i \in S_d$ there exists $m_j \in S_{\tilde{d}}$ such that $n_i |
m_j$ and vice versa. Then, for $n_{i_1} \in S_d,$ there exists
$m_{j_1} \in S_{\tilde{d}}$ such that $n_{i_1} | m_{j_1}$. Consider
$m_{j_1} \in S_{\tilde{d}}$, and similarly, there exists $n_{i_2}
\in S_d$ with $m_{j_1} | n_{i_2}$. By induction, we get infinite
subsets $L_d = \{n_{i_k} \} \subset S_d$ and $L_{\tilde{d}} =
\{m_{j_k} \} \subset S_{\tilde{d}}$ such that $H = L_d \cup
L_{\tilde{d}}$ can be a frequency integer set of some limit-periodic
potential $d^{'}$(see Remark~\ref{remk:limitperiod}). By Theorem
\ref{thm:subfrequency}, we conclude that $\Omega_d \cong
\Omega_{d^{'}} \cong \Omega_{\tilde{d}}$.

Conversely, assume that such a condition is not successful. Without
loss of generality, suppose that for a fixed $n_1 \in S_d$, $n_1
\nmid m_j$ for any $m_j \in S_{\tilde{d}}$. If $\Omega_d \cong
\Omega_{\tilde{d}}$, by the Pontryagin duality theorem, we have
$\hat{\Omega}_d \cong \hat{\Omega}_{\tilde{d}}$. So there is an
isomorphism $\phi : \hat{\Omega}_d \to \hat{\Omega}_{\tilde{d}}$
with $\phi(e^{2 \pi i}) = e^{2 \pi i}$ (the identity must be mapped
to the identity). Write
$$\phi(e^{\frac{2\pi}{n_1} i}) = e^{\sum^{t_1}_{j = 1} k_j \frac{2 \pi i}{m_j}} = e^{ q_1 \frac{2 \pi i}{m_{t_1}}},$$
and we get
$$\phi(e^{m_{t_1}\frac{2\pi}{n_1} i}) =  e^{ q_1 m_{t_1} \frac{2 \pi i}{m_{t_1}}} = e^{2 q_1 \pi i} = 1,$$ which is
impossible since $m_{t_1}/n_1$ is not an integer. Thus, $\Omega_d
\ncong \Omega_{\tilde{d}}.$ The proof is complete.
\end{proof}

By Theorem~\ref{thm:classification}, we conclude the following classification theorem.

\begin{theorem} \label{cor:hull}
Given two limit-periodic potentials $d,\tilde{d} \in
\ell^{\infty}(\Z)$, $\Omega_d \cong \Omega_{\tilde{d}}$ if and only
if $d$ and $\tilde{d}$ have the same maximal frequency integer set.
\end{theorem}

If $d \in \ell^{\infty}(\Z)$ is periodic, $\mathrm{hull}(d)$ is a
finite cyclic group. If $d$ is limit-periodic but not periodic, the
previous theorems don't give us any information about the structure
of the hull of a limit-periodic potential. We shall see that
$\mathrm{hull}(d)$ is a procyclic group later. We will also explain how to generate
limit-periodic potentials from a procyclic group.


\vspace*{0.5cm}
\setcounter{equation}{0}

\section{Profinite Groups}

\begin{definition}
A group $\Omega$ is called Cantor if it is a totally
disconnected compact Abelian topological group without isolated
points.
\end{definition}

\begin{remark}
Not every Cantor group has minimal translations. For example,
$$\Omega = \prod^{\infty}_{j=0} \Z_2,$$ where $\Z_2$ is a cyclic
$2$-group, is a Cantor group with the product topology, but it has
no minimal translations, since $\omega + \omega$ is the identity for
any $\omega \in \Omega$.
\end{remark}

By the definition of Cantor groups and Proposition
\ref{prop:avilahull}, we conclude that given a limit-periodic
potential $d \in \ell^{\infty}(\Z)$ with an infinite frequency
integer set (it ensures that $d$ is not periodic),
$\mathrm{hull}(d)$ is a Cantor group which has a minimal
translation. A topological profinite group is by definition an
\textit{inverse limit} of finite topological groups. Now let's
introduce the related definitions.

A \textit{directed} set is a partially ordered set $I$ such that for
all $i_1,i_2 \in I$ there is an element $j \in I$ for which $i_1
\leq j$ and $i_2 \leq j$.

\begin{definition}
An inverse system $(X_i,\phi_{ij})$ of topological groups indexed by
a directed set $I$ consists of a family $(X_i \mid i \in I)$ of
topological groups and a family $(\phi_{ij}: X_j \to X_i \mid i,j
\in I, i \leq j)$ of continuous homomorphisms such that $\phi_{ii}$
is the identity map $\mathrm{id}_{X_i}$ for each $i$ and  $\phi_{ij}
\phi_{jk} = \phi_{ik}$ whenever $i \leq j \leq k.$
\end{definition}

For example, let $I = \Z^{+}$ with the usual order, let $p$ be a
prime, and let $X_i = \Z/p^i \Z$ for each $i$, and for $j \ge i$ let
$\phi_{ij} : X_j \to X_i$ be the map defined by
$$\phi_{ij}(n + p^j \Z) = n + p^i \Z$$
for each $n \in \Z$. Then $(X_i, \phi_{ij})$ is an inverse system of finite topological groups.

\indent Let $(X_i, \phi_{ij})$ be an inverse system of topological
groups and let $Y$ be a topological group. We call a family $(\phi_i
: Y \to X_i \mid i \in I)$ of continuous homomorphisms
\textit{compatible} if $\phi_{ij} \phi_j = \phi_i$ whenever $i \leq
j$; that is, each of the following diagrams
\begin{diagram}
&              &Y &                &\\
&\ldTo^{\phi_j}&        &\rdTo^{\phi_i}   &\\
X_{j}      &        &\rTo^{\phi_{ij}} &      & X_{i}
\end{diagram}
is commutative.
\begin{definition}
An inverse limit $(X, \phi_i)$ of an inverse system $(X_i,
\phi_{ij})$ of topological groups is a topological group together
with a compatible family $(\phi_i : X \to X_i)$ of continuous
homomorphisms with the following universal property: whenever
$(\vphi_i : Y \to X_i)$ is a compatible family of continuous
homomorphisms from a topological group $Y$, there is a unique
continuous homomorphism $\vphi : Y \to X$ such that $\phi_i \vphi =
\vphi_i$ for each $i$.
\end{definition}

That is, there is a unique continuous homomorphism $\vphi$ such that
each of the following diagram

\begin{diagram}
&              &Y &                &\\
&\ldTo^{\vphi}&        &\rdTo^{\vphi_i}   &\\
X     &        &\rTo^{\phi_{i}} &      & X_{i}
\end{diagram}
is commutative.

\begin{proposition} \label{prop:inverselimit}\emph{\cite[Proposition
1.1.4]{w}}
Let $(X_i, \phi_{ij})$ be an inverse system of topological groups, indexed by $I$.\\[1mm]
{\rm (1)}. There exists an inverse limit $(X, \phi_i)$ of $(X_i,
\phi_{ij})$, for which $X$ is a topological group and the maps
$\phi_i$ are continuous
homomorphisms.\\[0.5mm]
{\rm (2)}. If $(X^{(1)}, \phi^{(1)}_i)$ and $(X^{(2)},
\phi^{(2)}_i)$ are inverse limits of the inverse system, then there
is an isomorphism
$\bar{\phi} : X^{(1)} \to X^{(2)} $ such that $\phi^{(2)}_i \bar{\phi} = \phi^{(1)}_i$ for each $i$. \\[0.5mm]
{\rm (3)}. Write $G = \prod_{i \in I} X_i$ with the product topology
and for each $i$ write $\pi_i$ for the projection map from $G$ to
$X_i$. Define
$$X = \{ c \in G : \phi_{ij} \pi_j(c) = \pi_i(c)\ for\ all\ i,j\ with\ j \ge i   \}$$
and $\phi_i = \pi_i |_{X}$ for each $i$. Then $(X,\phi_i)$ is an inverse limit of $(X_i, \phi_{ij})$.
\end{proposition}

The result above shows that the inverse limit of an inverse system
$(X_i,\phi_{ij})$ exists and is unique up to isomorphism \footnote{
We talk about topological groups throughout the paper, for which
\textit{isomorphism} means group isomorphism and topological space
homeomorphism, i.e., group isomorphism in addition should be
continuous.}.  We also have the following important characterization
of profinite groups.
\begin{proposition} \emph{\cite[Corollary 1.2.4]{w}}
Let $X$ be a topological group. The following are equivalent:\\[1mm]
{\rm (1)}. $X$ is profinite, i.e., it is an inverse limit of an inverse system;\\[0.5mm]
{\rm (2)}. $X$ is isomorphic to a closed subgroup of a product group of finite groups;\\[0.5mm]
{\rm (3)}. $X$ is compact and $\bigcap (N \mid N \lhd_{O} X) = 1$ ($\lhd_{O}$ means that $N$ is open and normal);\\[0.5mm]
{\rm (4)}. $X$ is compact and totally disconnected.
\end{proposition}

By the above proposition, we see that \textit{a Cantor group is an
Abelian profinite group without isolated points.}

\begin{proposition} \emph{\label{lem:subinverselimit}}
Assume that the directed set is $I = \Z^{+}$ with the usual order.
For an inverse system $(X_i, \phi_{ij})_{j \ge i}$ with the inverse
limit $(X,\phi_i)$, every non-finite sub-inverse system still has
the same inverse limit $(X,\phi_{i_k})$ up to isomorphism.
\end{proposition}

\begin{proof}
Consider a non-finite sub-inverse system $(X_{i_k}, \phi_{i_k
i_t})_{t \ge k }$. Assume that $(X^{'},\phi^{(1)}_{i_k})$ is the
inverse limit. Obviously, $(X,\phi_{i_k})$ is compatible with
$(X_{i_k}, \phi_{i_k i_t})_{t \ge k }$, so there is a unique
homomorphism $\phi^{(1)} : X \to X^{'}$ such that $\phi_{i_k} =
\phi^{(1)}_{i_k} \phi^{(1)}$.

For any $X_q$ not in the sub-inverse system, choose $i_k > q$. We
have that $\phi_{q i_k} : X_{i_k} \to X_q$ and $\phi^{(1)}_q =
\phi_{q i_k} \phi^{(1)}_{i_k} : X^{'} \to X_q$ are  homomorphisms.
We  will prove that $(X^{'}, \phi^{(1)}_i)$ is compatible with
$(X_i, \phi_{ij})_{j \ge i }$, for which it suffices to show that
the following diagram:

\begin{diagram}
            &                             & X^{'}                     &                      &                \\
            &\ldTo^{\phi^{(1)}_{i_k}}     & \dTo_{\phi^{(1)}_q}       &\rdTo^{\phi^{(1)}_{i_t}}         &\\
X_{i_k}     &\rTo_{\phi_{q i_k}}    & X_q                       &\rTo_{\phi_{i_t q}}    & X_{i_t}\
\end{diagram}
is commutative. The left half of the above diagram follows from the
definition of $\phi^{(1)}_q$. The right half follows from
$\phi^{(1)}_{i_t} = \phi_{i_t q} \phi_{q i_k} \phi^{(1)}_{i_k}
=\phi_{i_t q} \phi^{(1)}_q$. So $(X^{'}, \phi^{(1)}_i)$ is
compatible with $(X_i, \phi_{ij})_{j \ge i }$ and there is a unique
homomorphism  $\phi^{(2)} : X^{'} \to X$ such that $\phi^{(1)}_{i} =
\phi_i \phi^{(2)}.$

By the universal property for $(X^{'},\phi^{(1)}_{i_k})$, there is
only one map $F : X^{'} \to X^{'}$ with the property that
$\phi^{(1)}_{i_k} F = \phi^{(1)}_{i_k}$ for each $i$. However, both
$\phi^{(1)} \phi^{(2)}$ and $\mathrm{id}_{X^{'}}$ have this
property, so $\phi^{(1)} \phi^{(2)} = \mathrm{id}_{X^{'}}$.
Similarly, for $\phi^{(2)} \phi^{(1)}: X \to X$ we have $\phi_{i_k}
\phi^{(2)} \phi^{(1)} = \phi_{i_k}$. For any $\phi_j,$ $\phi_j =
\phi_{j i_k} \phi_{i_k}$ when $i_k > j$, so we have $\phi_j
\phi^{(2)} \phi^{(1)} = \phi_{j i_k} \phi_{i_k} \phi^{(2)}
\phi^{(1)} = \phi_{j i_k} \phi_{i_k} = \phi_j,$ which implies
$\phi^{(2)} \phi^{(1)} = \mathrm{id}_{X}$. Thus, $\phi^{(1)} : X \to
X^{'}$ is an isomorphism.
\end{proof}

\begin{proposition} \emph{\cite[Proposition 1.1.7]{w}}
\label{prop:procriterion} Let $G$ be a compact Hausdorff totally
disconnected space. Then $G$ is the inverse limit of its discrete
quotient spaces.
\end{proposition}

We interpret a $class$ in the usual sense that it is closed with
respect to taking isomorphic images. Let $\mathcal{C}$ be some class
of finite groups. We call a group $F$ a $\mathcal{C}$-group if $F
\in \mathcal{C}$, and $G$ is called a pro-$\mathcal{C}$ group if it
is an inverse limit of $\mathcal{C}$-groups. We say that
$\mathcal{C}$ is closed for quotients (resp. subgroups) if every
quotient group (resp. subgroup) of a $\mathcal{C}$-group is also a
$\mathcal{C}$-group. Similarly, we say that $\mathcal{C}$ is closed
for direct products if $F_1 \times F_2 \in \mathcal{C}$ whenever
$F_1 \in \mathcal{C}$ and $F_2 \in \mathcal{C}$. For example, for
the class of finite $p$-groups where $p$ is a fixed prime, an
inverse limit of finite $p$-groups is called a $pro$-$p$ group; for
the class of finite cyclic groups, an inverse limit of finite cyclic
groups is called a $procyclic$ group.

The next result describes how  a given profinite group, its
subgroups and quotient groups, can be represented explicitly as
inverse limits.

\begin{proposition}\emph{\cite[Theorem 1.2.5]{w}}
\label{prop:quotient} {\rm (1)}. Let $G$ be a profinite group. If
$I$ is a filter base of closed normal subgroups of $G$ such that
$\bigcap (N \mid N \in I) = 1$, then
$$G \cong {\lim_{\longleftarrow}}_{N \in I} G/N .$$
Moreover $$ H \cong {\lim_{\longleftarrow}}_{N \in I} H/(H \cap N) $$
for each closed subgroup $H$, and
$$ G/K \cong {\lim_{\longleftarrow}}_{N \in I} G/KN $$
for each closed normal subgroup $K$.\\[0.5mm]
{\rm (2)}. If $\mathcal{C}$ is a class of finite groups which is
closed for subgroups and direct products, then closed subgroups,
direct products and inverse limits of pro-$\mathcal{C}$ groups are
pro-$\mathcal{C}$ groups. If in addition $\mathcal{C}$ is closed for
quotients, then quotient groups of pro-$\mathcal{C}$ groups by
closed normal subgroups are pro-$\mathcal{C}$ groups.
\end{proposition}


\vspace*{0.5cm}
\setcounter{equation}{0}

\section{More Results about Limit-periodic Potentials}
\subsection{Generate limit-periodic potentials from a procyclic group.}
We shall see how to generate limit-periodic potentials from a
procyclic group, that is, a Cantor group which admits a minimal
translation.

\begin{proposition} \label{prop:hull} \emph{\cite[Lemma 2.2.]{a}}
Given a Cantor group $\Omega$ and a minimal translation $T$, for each $f \in C(\Omega,\R)$, define $F : \Omega \rarrow
\ell^\infty(\Z)$, $F(\om) = (f(T^n(\om)))_{n \in \Z}$. Then we have
that $F(\om)$ is limit-periodic and $F (\Omega) =
\mathrm{hull}(F(\om))$ for every  $\om \in \Omega$.
\end{proposition}

By Proposition \ref{prop:avilahull}, we know that
$\mathrm{hull}(F(e))$ ($e$ is the identity of $\Omega$) is a finite
cyclic group or it is a Cantor group with the unique group
structure: $\sigma^{0}(F(e)) = F(e)$ is the identity element, and
$\sigma^{i}(F(e)) \cdot \sigma^{j}(F(e)) = \sigma^{i+j}(F(e))$. The
translation $T$ defined by $T(\sigma^{i}(F(e))) =
\sigma^{i+1}(F(e))$ is minimal.

Since $C(\Omega, \R)$ can generate a class of limit-periodic
potentials, we will get certain class of topological groups by
taking the hulls of these limit-periodic potentials. The group in
this class is a Cantor group or a finite cyclic group.  What is the
relation between this class of topological groups and the original
Cantor group $\Omega$? To answer this question, we need two lemmas
first. (Note that we will adopt the notation $F(e)$ throughout this
section, that is, $F(e) = (f(T^n(e)))_{n \in \Z}$ as in
Proposition~\ref{prop:hull}.)

\begin{lemma} \label{lem:quotient}
For any $f \in C(\Omega,\R)$, $\mathrm{hull}(F(e))$ is a quotient
group of $\Omega$.
\end{lemma}

\begin{proof}
Define $\phi$ by $\phi: \Omega \longrightarrow \mathrm{hull}(F(e)),
\phi(\omega) = F(\omega).$ It is not hard to see that the group
structure of $\mathrm{hull}(F(e))$ is like $F(\omega_1)\cdot
F(\omega_2) = F(\omega_1 \cdot \omega_2),$ since $F(T^m(e)) =
\sigma^m (F(e))$.  It follows that $\phi$ is a group homomorphism.
By Proposition~\ref{prop:hull}, $\phi$ is surjective. The continuity
of $\phi$ follows from compactness of $\Omega$ and 
continuity of $f$. So we have
$$\mathrm{hull}(F(e)) \cong \Omega/\mathrm{ker}(\phi),$$
implying the lemma.
\end{proof}

\begin{lemma} \label{lem:hulliso}
There exists an $f \in C(\Omega,\R)$ such that $\mathrm{hull}(F(e)) \cong \Omega.$
\end{lemma}
\begin{proof}
By Lemma~\ref{lem:quotient}, it suffices to prove that  there exists
some $f \in C(\Omega,\R)$ such that $\mathrm{ker}(\phi) = \{ e \}$.
Clearly, a Cantor group is metrizable (recall that any separable
compact space is metrizable). Introduce a metric on $\Omega$
compatible with the topology. Define a function $f: \Omega \to \R$
by $f(\omega) = \mathrm{dist}(e,\omega).$ Clearly, $f$ is
continuous, so there is a corresponding $F$ (defined as in
Proposition~\ref{prop:hull}) $: \Omega \to \ell^{\infty}(\Z)$ such
that $\mathrm{hull}(F(e))$ is a quotient group of $\Omega$. Consider
$\phi: \Omega \longrightarrow \mathrm{hull}(F(e)), \phi(\omega) =
F(\omega).$ If $F(\omega) = F(e)$, then $f(\omega) = f(e)$, that is,
$\mathrm{dist}(e,\omega) = \mathrm{dist}(e,e) = 0$, implying $\omega
= e$ and $\mathrm{ker}(\phi) = \{ e \}$.
\end{proof}

\begin{theorem} \label{thm:quocantor}
Given a Cantor group $\Omega$ and a minimal translation $T$, for
each $f \in C(\Omega,\R)$, $\mathrm{hull}(F(e))$ is a Cantor group
or  a finite cyclic group, and $C(\Omega,\R)$ can generate a class
of topological groups.  Then there is a one-to-one correspondence
between this class of topological groups and quotient groups of
$\Omega$.
\end{theorem}

\begin{proof}
By Lemma \ref{lem:quotient}, we know that groups in this class are
quotient groups of $\Omega$. Let's show the converse direction. By
the definition, a quotient group of a Cantor group is still Cantor
or finitely cyclic. Given a quotient group $\Omega_0$ and a quotient
homomorphism $q: \Omega \to \Omega_0$,  we claim that  $T$ will
induce a minimal translation $T_0$ on $\Omega_0$ such that
$T_0([\omega]) = q(T(\omega)), [\omega] \in \Omega_0$. For writing
convenience, we assume that the group operation is addition and
$T(\omega) = \omega + \omega_1$. If $[\omega] = [\omega^{'}]$, then
$T_0([\omega]) = q(T(\omega)) = q(\omega_1 + \omega) = q(\omega_1) +
q(\omega) = [\omega_1] + [\omega] = [\omega_1] + [\omega^{'}] =
T_0([\omega^{'}])$, which gives that $T_0$ is a translation by
$[\omega_1]$. That $T_0$ is minimal follows from the fact that $T$
is minimal and $q$ is continuous. By Lemma \ref{lem:hulliso}, we
know that for $\Omega_0$ and $T_0$ there exists some $f_0 \in
C(\Omega_0,\R)$ such that $\mathrm{hull}((f_0(T^{n}_0 ([e])))_{n \in
\Z}) \cong \Omega_0.$ Let $f = f_0 \circ q$. Clearly, $f \in
C(\Omega,\R)$ and the following diagram

\begin{diagram}
                           & \Omega                     &                      &                \\
               & \dTo_{q}       &\rdDashto^{f}         &\\
 & \Omega_0                       &\rTo_{f_0}    &  \R
\end{diagram}
is commutative. So we have $\mathrm{hull}(F(e)) \cong
\mathrm{hull}((f_0(T^{n}_0 ([e])))_{n \in \Z}) \cong \Omega_0.$ The
proof is complete.
\end{proof}

We also have

\begin{theorem} \label{thm:iso}
Given a Cantor group $\Omega$ and a minimal translation $T$, for any
limit-periodic potential $d \in \ell^\infty(\Z)$ satisfying
$\mathrm{hull}(d) \cong \Omega$, there is an $f \in C(\Omega,\R)$
such that $f(T^n(e)) = d_n$ for every $n \in \Z$.
\end{theorem}

\begin{proof}
By Lemma \ref{lem:hulliso} we have $\tilde{f} \in C(\Omega, R)$ such
that $\mathrm{hull}(\tilde{F}(e)) \cong \Omega$ (note that
$\tilde{F}(e) = (\tilde{f}(T^n(e)))_{n \in \Z}$). Since
$\mathrm{hull}(d) \cong \Omega,$ we have a continuous isomorphism $h
: \mathrm{hull}(\tilde{F}(e)) \to \mathrm{hull}(d)$ with
$h(\tilde{F}(e)) = d.$

Clearly, for $T^{n_k}(e) \in \Omega$ we have
$h(\tilde{F}(T^{n_k}(e))) = \sigma^{n_k}(d)$ since
$\tilde{F}(T^{n_k}(e)) = \sigma^{n_k}(\tilde{F}(e)).$ If $\lim_{k
\to \infty} T^{n_k}(e) = \omega,$ then $h(\tilde{F}(\om)) =
\lim_{k\to \infty}\sigma^{n_k}(d)$, where the limit exists since $h$
and $\tilde{F}$ are both continuous. Define $f$ by $f(T^n(e)) =
\sigma^n(d)_0 = d_n$. We extend $f$ to the whole $\Omega$ by
$f(\omega) = \lim_{k \to \infty} \sigma^{n_k}(d)_0$ if $\omega =
\lim_{k \to \infty} T^{n_k}(e).$ By the previous analysis, $f$ is
well defined and continuous. So there is an $f \in C(\Omega,\R)$
such that $F(e) = d,$ that is,  $f(T^n(e)) = d_n$ for every $n \in
\Z$.
\end{proof}

We say that $f \in C(\Omega,\R)$ is $p$-periodic with respect to $T$
if $f(T^p(\omega)) = f(\omega)$ for every $\omega \in \Omega$. Then
we have

\begin{proposition} \emph{\cite[Proposition 2.7]{dg2}}
Let $f \in C(\Omega,\R)$. If $f(T^{p+m}(\omega_0)) =
f(T^{m}\omega_0)$ for some $\omega_0 \in \Omega$, some minimal
translation $T : \Omega \to \Omega$ and every $m \in \Z$, then for
every minimal translation $\tilde{T} : \Omega \to \Omega$, $f$ is
$p$-periodic with respect to $\tilde{T}$.
\end{proposition}

The above proposition tells us that the periodicity of $f$ is
independent of minimal translations. Next we recall from \cite{a}
how periodic functions in $C(\Omega,\R)$ can be constructed. Given a
compact open subgroup $\Omega_0$ with finite index (such open
subgroups can be found in any neighborhood of the identity element;
see Proposition~\ref{prop:avilahull}) and $f \in C(\Omega,\R)$, we
can define a periodic $f_{\Omega_0} \in C(\Omega,\R)$ by
$$
f_{\Omega_0}(\omega) = \int_{\Omega_0} f(\omega \cdot \tilde \omega) \, d\mu_{\Omega_0}(\tilde \omega),
$$
where $\mu_{\Omega_0}$ denotes the normalized Haar measure on $\Omega_0$. This
shows that the set of periodic functions is dense in
$C(\Omega,\R)$ since $\Omega_0$ can be arbitrarily small. Moreover, there
exists a decreasing sequence of Cantor subgroups $\Omega_k$ with
finite index $n_k$ such that $\bigcap \Omega_k = \{e\}$ (we will see this point explicitly in the next subsection).
Let $P_k$ be the set of functions defined on $\Omega / \Omega_k$, that is, the
elements in $P_k$ are $n_k$-periodic potentials. Denote by $P$
the set of all periodic functions. Then, we have $P_{k}
\subset P_{k+1}$  and $P =
\bigcup P_k$.

\subsection{Characterization of the hull of a limit-periodic potential}

Given a limit-periodic potential $d \in \ell^{\infty}(\Z)$ with an
infinite frequency integer set $S_d = \{n_j \},$ we see that
$\mathrm{hull}(d)$ is a Cantor group admitting a minimal
translation. Consider the directed set $I = \Z^{+}$ with the usual
order. This gives rise to an inverse system $(\Z_{n_i},\pi_{ij})_{j
\ge i},$ where $\Z_{n_j}$ are $n_j$-cyclic groups with the discrete
topology and $\pi_{ij}$ is a homomorphism defined by $\pi_{ij} (k +
n_j \Z) = k + n_i \Z,\ k \in \Z .$

We endow $\Z_{n_j}$ with a discrete metric defined by
$\mathrm{dist}_j(a_1,a_2) = 0$ when $a_1 = a_2$ and
$\mathrm{dist}_j(a_1,a_2) = 1$ when $a_1 \neq a_2$. Consider the
product group $$A = \prod^{\infty}_{j=1} \Z_{n_j},$$ of which the
topology is the product topology. We endow $A$ with a metric defined
by \be \label{metric} \mathrm{dist}(x,y) = \sum^{\infty}_{j = 1}
\frac{1}{2^j} \frac{\mathrm{dist}_j(x_j,y_j)}{1 +
\mathrm{dist}_j(x_j,y_j)}, \quad x,\ y \in A, \ee which is
compatible with the product topology.

Let $E=(1,1,\cdots,1,\cdots) \in A$, and consider the closed
subgroup
$$\bar{B} = \overline{\{ nE=(n,n,\cdots,n,\cdots) \in A : n \in \Z
\}}.$$ Obviously, $\bar{B}$ is a Cantor group with a minimal
translation $T(x) = x + E$, and $\vec{0} \in \bar{B}$ is the
identity element. Let
$$\tilde{d}_k = \mathrm{dist}(kE, \vec{0}) = \sum^{\infty}_{j = 1} \frac{1}{2^j} \frac{\mathrm{dist}_j(k,0)}{1 + \mathrm{dist}_j(k,0)}.$$
By the proof of Lemma \ref{lem:hulliso}, we know that $\tilde{d} =
(\tilde{d}_k)_{k \in \Z}$ is limit-periodic and
$$\mathrm{hull}(\tilde{d}) \cong \bar{B}.$$ Let $p_j(k) = \frac{1}{2^j} \frac{\mathrm{dist}_j(k,0)}{1 + \mathrm{dist}_j(k,0)}$.
Then we have  $$\tilde{d}(k) = \sum^{\infty}_{j = 1} p_j(k),$$ which
tells us that one of $\tilde{d}$'s frequency integer sets  is
$S_d = \{ n_j \}$. By Theorem \ref{thm:subfrequency}, we conclude
that
$$\mathrm{hull}(d) \cong \mathrm{hull}(\tilde{d}) \cong \bar{B}.$$

Let $\bar{B}_k = \overline{\{nn_kE : n \in \Z \}} \subset \bar{B}$,
and it is easy to  see that there exists a decreasing sequence of Cantor subgroups
$\bar{B}_k$ with the index $n_k$ and $\bigcap \bar{B}_k = \{ \vec{0}
\}$.

\begin{proposition}
$b = (k, k, \cdots, k, \cdots)$ is a generator in $\bar{B}$, that is, $\{ nb : n \in \Z \}$ is dense in $\bar{B},$
if and only if, for any $n_j$, $k$ and $n_j$ have no common divisors.
\end{proposition}

\begin{proof}
If there exists some $n_t$ such that $(k, n_t) = k_t > 1$, then $k$
cannot be a generator for $Z_{n_t}$, i.e. $ n k \neq 1
(\mathrm{mod}\ n_t)$ for any $n \in \Z$, and then $$\inf_{n \in \Z}
\p n b - E \p \ge \frac{1}{2^t} \frac{\mathrm{dist}_t(nk,1)}{1 +
\mathrm{dist}_t(nk,1)} = \frac{1}{2^{t+1}},$$ where
$E=(1,1,\cdots,1,\cdots)$. So $b$ is not a generator.

Conversely, if for any $n_j$, $k$ and $n_j$ have no common divisors,
i.e. $(k,n_j) = 1$, we will show that there exists a positive
integer sequence $\{ q_j \}_{j \in \Z^{+}}$ such that $\lim_{j \to
\infty} q_j b = E$ in the norm sense.  Since $(k,n_1) = 1,$ there
exists some positive integer $q_1$ such that $ q_1 k = 1\
(\mathrm{mod}\ n_1).$ Similarly, since $(k,n_2) = 1,$ we have $ q_2
k = 1\ (\mathrm{mod}\ n_2).$ Since $n_1 | n_2$, we still have $ q_2
k = 1\ (\mathrm{mod}\ n_1)$. By induction, we get a sequence $\{ q_j
\}_{j \in \Z^{+}}$ such that $q_j k = 1 \ (\mathrm{mod}\ n_i)$ when
$i \leq j$. It is easy to see that $\lim_{j \to \infty} q_j b = E$
in the norm sense, and so $b$ is a generator.
\end{proof}

\begin{proposition}
Let $T : \bar{B} \to \bar{B}$,  $ x \to  b + x$. $T$ is minimal if and only if $b$ is a generator.
\end{proposition}
\begin{proof}
If $b$ is not a generator, clearly $T$ is not minimal. On the other
hand, if $b$ is a generator, for each $x \in \bar{B}$, there exists
a corresponding sequence $\{ q_j \}_{j \in \Z^{+}}$ such that
$\lim_{j \to \infty } q_j b = E - x$ ($E$ is the same as before),
and so $\lim_{j \to \infty } (q_j b + x) = E$, which implies
$\lim_{j \to \infty} T^{q_j}(x) = E$. It follows  that $T$ is
minimal.
\end{proof}

Furthermore, since the following diagram

\begin{diagram}
&             &\bar{B} &                &\\
&\ldTo^{\pi_j}&        &\rdTo^{\pi_i}   &\\
\Z_{n_j}      &        &\rTo^{\pi_{ij}} &      & \Z_{n_i}
\end{diagram}
is commutative where $\pi_i$ is the $i$-th coordinate projection,
$(\bar{B},\pi_i)$ is compatible with this inverse system
$(\Z_{n_i},\pi_{ij})_{j \ge i}$. Proposition \ref{prop:inverselimit}
ensures that $(\bar{B},\pi_i)$ is also the inverse limit of this
system (see the statement (3) of Proposition
\ref{prop:inverselimit}). As already introduced, we call $\bar{B}$ a
procyclic group, that is, an inverse limit of finite cyclic groups.
Equivalently, a procyclic group is a profinite group that can be
generated by one element.  Thus, we have proved that \textit{a
Cantor group with a minimal translation is a procyclic group}. Also,
equivalently, we have the following.

\begin{theorem}
For any limit-periodic potential $d \in \ell^{\infty}(\Z)$, $\mathrm{hull}(d)$ is a procyclic group.
\end{theorem}

By the above analysis and Proposition \ref{lem:subinverselimit},
Theorem \ref{thm:subfrequency} follows directly. There is also a
classification theorem about procyclic groups, which is essentially
the same as Theorem \ref{thm:classification}. Let $n = \prod_{p}
p^{n(p)}$ be a supernatural number where $p$ goes through all the
primes with $0 \leq n(p) \leq \infty$ (one can consider this
expression as the generalized prime factorization, and we call it
$supernatural$ since it is an extension of natural numbers to
infinities that differ by different factorizations).
\begin{theorem} \emph{\cite[Theorem 2.7.2]{rz}}
There exists a unique procyclic group $G$ of order $n$ up to isomorphism.
\end{theorem}

\begin{remark}
(1). If $n= p^{n(p)}$, where $p$ is a prime and $n(p) = \infty$,
then the associated procyclic group is the group of $p$-adic
integers (see p.26 of \cite{rz}).\\[0.5mm]
(2). Given a limit-periodic potential $d \in \ell^\infty(\Z)$,
$\mathrm{hull}(d)$ has order $n$. If $S_d = \{ n_j \}$ is a
frequency integer set of $d$, then one must have $\lim_{j \to
\infty} n_j = n$ (here ``$=$" means that they have the same
 generalized prime factorization).
\end{remark}

At the end of this section, we discuss quotient groups of a
procyclic group. The class of cyclic groups is closed with respect
to quotients, that is, a quotient group of a cyclic group is still
cyclic. By Proposition \ref{prop:quotient}, we know that a quotient
group of a procyclic group is still procyclic. We have

\begin{proposition}
Given a procyclic group $G$ with order $n = \prod_{j \in \Z^+}
p^{r_j}_j$ where $p_j$ are primes with $0 < r_j \leq \infty$,
quotient groups of $G$ are procyclic groups with order $m = \prod_{j
\in \Z^+} p^{a_j}_j$ where $0 \leq a_j \leq r_j$, and vice versa.
\end{proposition}

\begin{proof}
Obviously, $G \cong \lim_{ \leftarrow k} Z_{n_k}$, where $n_k |
n_{k+1}$ and $\lim_{k \to \infty} n_k = n.$ Write $\bar{B} =
\overline{\{n E : n \in \Z \}}$, where $E = (1,1,\cdots,1,\cdots)
\in \prod_{k } Z_{n_k}.$ Then, $\bar{B} \cong \lim_{ \leftarrow k}
Z_{n_k}.$ It is sufficient to consider $\bar{B}$.

Write $\bar{B}_k = \overline{\{n n_k E : n \in \Z \}}$. Clearly, $\{
\bar{B}_k \}$ is a decreasing sequence of open (also closed)
subgroups of $\bar{B}$ with index $n_k$ and $\bigcap \bar{B}_k = \{
\vec{0} \}$. By Proposition \ref{prop:quotient},  for every closed
subgroup $N \subset \bar{B}$ (every subgroup is normal in an Abelian
group), we have $\bar{B}/N \cong \lim_{ \leftarrow k}
\bar{B}/\bar{B}_kN .$ Since $\bar{B}/\bar{B}_k$ is an $n_k$-cyclic
group, $\bar{B}/\bar{B}_kN$ is an $m_k$-cyclic group with $m_k |
n_k$. $B/N$ is a procyclic group with order $m = \lim_{k \to \infty}
m_k$. Since $m_k | n_k$, it follows that $m = \prod_{j \in \Z^+}
p^{a_j}_j$ where $0 \leq a_j \leq r_j$.

Conversely, if $m = \prod_{j \in \Z^+} p^{a_j}_j$ where $0 \leq a_j
\leq r_j$, obviously there exist $m_k$ such that $m_k | n_k$ and
$\lim_{k \to \infty} m_k = m$. Let $\tilde{B} = \lim_{ \leftarrow k}
Z_{m_k}$, and define $\phi : \bar{B} \to \tilde{B},$ $\phi(E) =
\tilde{E}$, where $\tilde{E} = (1,1,\cdots,1,\cdots) \in \prod_{k}
Z_{m_k}$. Metrics on $\tilde{B}$ and $\bar{B}$ are the metric like
(\ref{metric}). It is easy to see that $\phi$ is a continuous
surjective homomorphism. So $\tilde{B}$ is a quotient group of
$\bar{B}$.
\end{proof}

\begin{remark}
By the above proposition, it is easy to see that there exists a 
universal procyclic group with order $n = \prod_p
p^{\infty}$, where $p$ goes though all the primes, in the sense that any procyclic group is a quotient group of this group.
\end{remark}


\vspace*{0.5cm}
\setcounter{equation}{0}

\section{Applications in the Schr\"odinger Operators}
In this section, we will review some recent  results on
the limit-periodic Schr\"odinger operators obtained in this context.
First, let us describe the  model. We consider
Schr\"odinger operators $H^\omega_{f,T}$ acting in $\ell^2(\Z)$ with
dynamically defined potentials $V^\omega_{f,T}$ given by
\begin{equation}\label{equ:oper}
[H^\omega_{f,T} u](n) = u(n+1) + u(n-1) + V^\omega_{f,T}(n) u(n),
\end{equation}
where
\begin{equation}\label{equ:pot}
V^\omega_{f,T} (n) = f(T^n (\omega)) , \quad \omega \in \Omega, \; n \in
\Z
\end{equation}
with a homeomorphism $T$ of a compact space $\Omega$ and a
continuous sampling function $f : \Omega \to \R$.

We have seen that once $\Omega$ is a Cantor group and $T$ a minimal
translation, $V^\omega_{f,T} (n) = f(T^n (\omega))$ is a
limit-periodic potential. Clearly,  $T$ is ergodic with respect to
$\mu$,  the normalized Haar measure on $\Omega$. The
ability to fix the base dynamics and independently vary the sampling
functions is very useful in constructing examples of Schr\"odinger
operators with  desired spectral features.

\begin{theorem} \emph{\cite[Theorem 1.1]{a}}
Suppose $\Omega$ is a Cantor group and $T$ a minimal translation.
For a dense set of $f \in C(\Omega,\R)$ and every $\lambda \neq 0$,
the spectrum of $H^\omega_{\lambda f,T}$ has zero Lebesgue measure,
and the Lyapunov exponent is a continuous positive function of the
energy. 
\end{theorem}

By the above theorem, Avila gave first examples of ergodic potentials with a spectrum of zero Lebesgue measure such that
the Lyapunov exponent  is positive throughout the spectrum. This answers a question raised by Simon (Conjecture 8.1 of \cite{s}).

\begin{theorem} \emph{\cite[Theorem 1.1]{dg1}}
Suppose $\Omega$ is a Cantor group and $T$
a minimal translation. We have \\[1mm]
{\rm (1)} For a dense set of $f \in C(\Omega,\R)$ and $\lambda \neq 0$, the spectrum of $H^\omega_{\lambda f,T}$ is
a Cantor set of positive Lebesgue measure and purely absolutely continuous.  \\[0.5mm]
{\rm (2)} For a dense $G_\delta$-set of $f \in C(\Omega,\R)$, every $\lambda \neq 0$ and every $\omega \in \Omega$, 
the spectrum of $H^\omega_{\lambda f,T}$ is a Cantor set of zero Lebesgue measure and purely singular continuous. \\[0.5mm]
\end{theorem}

The proof of the above theorem shows that the Lyapunov exponent
vanishes throughout the spectrum. This gives  spectral information
about the limit-periodic potentials in the regime of zero Lyapunov
exponents.

\begin{definition} \label{d.gordon}
A bounded map $V: \Z \rightarrow \R$ is called a Gordon potential if
there exist positive integers $q_j \rightarrow \infty$ such that
$$
\max_{1\leq n \leq q_j}\left|V(n)-V(n\pm q_j) \right|\leq j^{-q_j}
$$
for every $j \ge 1$.
\end{definition}

\begin{theorem} \emph{\cite[Theorem 1.3]{dg2}} \label{thm:sc}
Suppose $\Omega$ is a Cantor group and $T$ a minimal translation.
Then there exists a dense set $\mathcal{F} \subset C(\Omega,\R)$
such that for every $f \in \mathcal{F}$, every $\lambda \neq 0$ and every $\omega \in \Omega$,
the following statements hold true: the spectrum of
$H^\omega_{\lambda f,T}$ has zero Hausdorff dimension, it is purely
singular continuous, and $E \mapsto L(E,T,\lambda f)$ is a positive
continuous function. 
\end{theorem}

\begin{theorem} \emph{\cite[Theorem 1.5]{dg2}} \label{thm:gordon}
Suppose $\Omega$ is a Cantor group and $T$ a minimal translation.
Then there exists a dense set $\mathcal{F} \subset C(\Omega,\R)$
such that for every $f \in \mathcal{F}$, every minimal translation
$T : \Omega \rightarrow \Omega$, every $\omega \in \Omega$, and
every $\lambda \neq 0$, $\lambda f(T^n(\omega))$ is a Gordon
potential.
\end{theorem}

Theorem \ref{thm:sc} and Theorem \ref{thm:gordon} are both from
\cite{dg2}. The Gordon lemma ensures that the Schr\"odinger operators
with Gordon potentials have no point spectrum. The proof of Theorem
\ref{thm:sc} used Theorem \ref{thm:gordon} to conclude absence of
point spectrum. Absence of $a.c.$ spectrum follows from zero
Hausdorff dimension. Theorem \ref{thm:sc} gives first examples of
limit-periodic potentials with purely singular continuous spectrum
in the regime of positive Lyapunov exponents.

Also, recent work of Damanik and Gorodetski, announced in \cite{dg},
focused on the weakly coupled Fibonacci Hamiltonian, that is an
ergodic model not (uniformly) almost periodic. One of the theorems
in \cite{dg} says that the Hausdorff dimension of the spectrum, as a
function of the coupling constant, is continuous at zero. In
contrast to this, Theorem \ref{thm:sc} tells us that there are
limit-periodic potentials such that continuity at zero coupling
fails since the Hausdorff dimension of the spectrum is zero for all
non-zero values of the coupling constant.

\begin{definition}
We say that a family $\{u_j\} \subset \ell^2(\Z)$ is uniformly
localized if there exist constants $r > 0$, called the decay rate,
and $c < \infty$ such that for every element $u_j$ of the family,
one can find $m_j \in \Z$, called the center of localization, so
that $|u_j(n)| \leq c e^{-r|n-m_j|}$ for every $n \in \Z$. We say
that the operator $H_\omega$ has ULE if it has a complete set of
uniformly localized eigenfunctions.\footnote{Recall that a set of
vectors is called complete if their span (i.e., the set of finite
linear combinations of vectors from this set) is dense.}
\end{definition}

For any Cantor group $\Omega$ that admits a minimal translation $T$,
there exists a limit-periodic potential $d$  such that
$\mathrm{hull}(d) \cong \Omega$. Write $M_d$ as the maximal
frequency integer set of $d$. We also say that $M_d$ is the maximal
frequency integer set of $\Omega$. Theorem \ref{cor:hull} ensures
that it is well defined.

\begin{definition}
Given a Cantor group that admits a minimal translation, we say that
it satisfies the condition $\mathscr{A}$ if its maximal frequency
integer set $M = \{m_j\} \subset \Z^+$ has the following property:
there exists some integer $m \ge 2$ such that for every $j$, we have
$m_j < m_{j+1} \leq m^m_j,$ that is, $\log m_{j+1}/\log m_j$ is
uniformly bounded.
\end{definition}

\begin{theorem} \emph{\cite[Theorem 2.5]{dg3}} \emph{\label{thm:uniformlocal}}
Suppose that $\Omega$ satisfies the condition $\mathscr{A}$. Then
there exists an $f \in C(\Omega,\R) $ such that for every $\om \in
\Omega$, $H^\omega_{f,T}$ has the same pure point spectrum and ULE
with $\omega$-independent constants.
\end{theorem}

The authors of \cite{dg3} conclude this result by applying P\"oschel'
general theorem \cite{p} to the limit-periodic potentials. Since
localization is a topic that has been explored in the context of
Schr\"odinger operators to a great extent, we would like to talk
more about this result.

In the topological setting, if $T$ is minimal and $f$ continuous,
the spectrum and the $a.c.$ spectrum are independent of $\om$;
however, the point spectrum and the $s.c.$ spectrum are in general
not. If pure point spectrum is independent of $\om$, we call it
\textit{phase stable}. This phenomenon is unusual since most of  known
models are not phase stable. For example, it is well known that pure
point spectrum cannot be independent of $\om$ in quasi-periodic
Schr\"odinger operators \cite{js}; for the Anderson model, one can
find periodic sequences $\om$ making the spectrum absolutely
continuous. Even for the Maryland model, the Schr\"odinger operator
(\ref{equ:oper}) with $\Omega = S^1$ and $V_{\om}(n) = \lam
\tan{(\om + 2 \pi n \alpha)}$, where $\alpha$ is a diophantine
irrational, we only know that for every $\om \in \Omega \setminus
\{\frac{\pi}{2} + \Z + \alpha \Z\}$ and any $\lam \neq 0$, the
Schr\"odinger operator (\ref{equ:oper}) has a pure point spectrum
\cite{fp,fg,pg}. Theorem \ref{thm:uniformlocal} gives the first
example whose spectrum is phase stable (modulo some unpublished work
of Jitomirskaya, which established a similar result for a bounded non-almost
periodic potential). Furthermore, the example has ULE for every $\omega
\in \Omega$, which is  stronger than phase stability (please refer
to \cite{j}).

There are two open questions. Is there ULE in a procyclic group of
other type, i.e., a Cantor group with minimal translations that
don't satisfy the condition $\mathscr{A}$? Given a Cantor group
$\Omega$ that admits a minimal translation, does there exist a dense
set of $f \in C(\Omega, \R)$ such that the  spectrum is
pure point?

\begin{theorem} \emph{\cite[Theorem 1.2]{gk}} \label{thm:ids}
Given any increasing continuous function $\varphi : \R^{+} \to
\R^{+}$ with \be \lim_{x\to 0} \varphi(x) = 0 \ee and a constant
$C_0 > 0$, there is a limit-periodic $V$ satisfying $\|V\|_{\infty}
\leq C_0$ such that the associated integrated density of states satisfies \be
\label{equa:thm} \limsup_{E \to E_0} \frac{|k_V(E)  - k_V(E_0)|
\log(|E - E_0|^{-1})}{\varphi(|E - E_0|)} = \infty, \ee for any $E_0
\in \sigma(\Delta  + V)$.
\end{theorem}

The above limit-periodic potential was also constructed through a
Cantor group and a minimal translation. The authors \cite{gk}
heavily use \cite[Lemma 3.1, Lemma 3.2]{a} to conclude
Theorem~\ref{thm:ids}. Craig and Simon in \cite{cs} showed that
integrated density of states are in general log H\"older continuous.
Theorem \ref{thm:ids} tells us that continuity of integrated
density of states cannot be improved for all potentials beyond log
H\"older continuity. These limit-periodic potentials are dense in
the space of limit-periodic potentials.

\section*{Acknowledgements}
I am deeply grateful to David Damanik, who introduced me to this
topic and encouraged me to write this expository paper. Also, his
comments have been very helpful for completion of the paper. I also
thank the referee for useful suggestions on how to improve the
presentation.


\end{document}